\documentclass[a4paper, 12pt]{article}
\usepackage{}

\usepackage{mathrsfs}
\usepackage{epsfig}
\usepackage{amsmath}
\usepackage{amssymb}
\usepackage{latexsym}
\usepackage{amsfonts}
\usepackage{amsthm}

\usepackage[numbers,sort&compress]{natbib}
\usepackage{graphicx}
\ifx\pdfoutput\undefined  \else
 \fi

\usepackage[centerlast]{caption2}

\usepackage{color}

\usepackage{txfonts}
\usepackage{amssymb}
\usepackage{mathrsfs}
\usepackage{amsfonts}

\newtheorem{theorem}{Theorem}[section]
\newtheorem{lemma}[theorem]{Lemma}

\newtheorem{prop}[theorem]{Proposition}
\newtheorem{conj}[theorem]{Conjecture}

\newtheorem{exam}[theorem]{Example}
\newtheorem{definition}[theorem]{Definition}

\usepackage{latexsym,amssymb}
\usepackage{amsmath}

\begin{document}

\title{{Weighted Erd\H{o}s-Burgess and Davenport constant in commutative rings}}
\author{
Guoqing Wang\\
\small{School of
Mathematical Sciences, Tiangong University, Tianjin, 300387, P. R. China}\\
\small{Email: gqwang1979@aliyun.com}
\\
}
\date{}
\maketitle

\begin{abstract}
Let $R$ be a finite commutative unitary ring. An idempotent in $R$ is an element $e\in R$ with $e^2=e$. Let $\Psi$ be a subgroup of the group ${\rm Aut}(R)$ of all automorphisms of $R$. The $\Psi-$weighted Erd\H{o}s-Burgess constant ${\rm I}_{\Psi}(R)$ is defined as
the smallest positive integer $\ell$  such that every sequence over $R$
of length at least $\ell$ must contain a nonempty subsequence $a_1,\ldots, a_{r}$
such that $\prod\limits_{i=1}^r  \psi_i(a_i)$ is one idempotent of $R$ where $\psi_1,\ldots,\psi_r\in \Psi$. In this paper, for the finite quotient ring of a Dedekind domain $R$, a connection is established between the $\Psi-$weighted-Erd\H{o}s-Burgess constant of $R$ and the $\Psi-$weighted Davenport constant of its group of units by all the prime ideals of $R$.

\end{abstract}

\section {Introduction}

The additive properties of sequences in finite abelian groups have been widely studied (see \cite{GaoGeroldingersurvey} for a survey), since  K. Rogers \cite{rog1} in 1963 pioneered the investigation of a combinatorial invariant associated with an arbitrary finite abelian group $G$, here denoted by ${\rm D}(G)$ and called the Davenport constant of $G$, which can be defined as the smallest $\ell\in \mathbb{N}$ such that
every sequence $T$ of terms from the group $G$ of length at least $\ell$ contains a nonempty subsequence $T'$ the sum of whose terms is equal to the identity element of the group $G$. Later on, H. Davenport popularized this invariant in the study of algebraic number theory (as reported in \cite{Olson1}), which may be the reason why this invariant was named after Davenport instead of its pioneer, Rogers.  The invariant ${\rm D}(G)$ has been investigated extensively in the past almost 60 years, and also been generalized from different aspects in connection with Factorization Theory in Algebra or binary quadratic forms, etc (see \cite{Halter-Koch, DeMaOrOr,Skalba, SkalbaEuropean}).
Adhikari, Chen  et.al  \cite{AdiChenDis,AdiChenJCTA, AdiRath} defined the Davenport constant with integers weights, which motivates a huge amount of researches (see \cite{AdiDaivSun, Grynkiewicz, Luca, MaOrsaSchmid, Thangadurai, YuanZengEruo}, and \cite{ Grynkiewiczmono} Chapter 16). Zeng and Yuan \cite{ZengYuanDiscrete} further defined the Davenport constant with homomorphisms weights as follows.

\noindent \textbf{Definition A.} \cite{ZengYuanDiscrete} \ {\sl Let $G^*, G$ be finite abelian groups written additively with a nontrivial homomorphism group ${\rm Hom}(G^*, G)$. For a nonempty subset $\Psi$ of ${\rm Hom}(G^*, G)$, let ${\rm D}_{\Psi}(G)$ be the least positive integer $\ell$
such that any sequence over $G^*$ of length
least $\ell$  must contain some terms $a_1,\ldots, a_{r}$ ($r\geq 1$)
such that $\sum\limits_{i=1}^{r}  \psi_i(a_i)=0_G$ for some $\psi_1,\ldots,\psi_{r}\in \Psi$.}

Recently, the Davenport constant together with some other additive properties of sequences were investigated in the setting of semigroups and rings (see \cite{Deng, wangDavenportII, wangAddtiveirreducible, wangidempotent, wangAIMS, wangINJ, wangAdmath, wangHaoZhuang, wangZhangWangQu} for example.) Among them, one of the following combinatorial invariant on idempotents of semigroups was proposed due to a question by P. Erd\H{o}s (see \cite{Burgess69, Gillam72}).

\noindent \textbf{Definition B.} (\cite{wangidempotent}, Definition 4.1) \ {\sl For a commutative semigroup $\mathcal{S}$, define the {\bf Erd\H{o}s-Burgess constant} of $\mathcal{S}$, denoted by $\textsc{I}(\mathcal{S})$, to be the least $\ell\in \mathbb{N}\cup \{\infty\}$ such that every sequence $T$ of terms from $\mathcal{S}$ and of length $\ell$ must contain one or more terms whose product is an idempotent.}

For a finite abelian group $G$, since the identity is the unique idempotent,  then
the Erd\H{o}s-Burgess constant of $G$ reduces to the Davenport constant. That is, the Erd\H{o}s-Burgess constant is the natural generalization of Davenport constant into semigroups. While, the most important class of commutative semigroups are the multiplicative semigroups of commutative rings.
On commutative rings, the author \cite{wangAdmath} showed that the Erd\H{o}s-Burgess constant exists only for {\bf finite} commutative rings except for a family of infinite commutative rings with a given very special form. That is, to study this invariant in the realm of commutative rings, we may consider it only for {\sl finite} commutative rings. The upper bound for a general finite semigroup can be given by the Gillam-Hall-Williams Theorem \cite{Gillam72}. Recently, the lower bound of Erd\H{o}s-Burgess constant was obtained for some classical finite commutative principal ideal rings (see \cite{wangHaoZhuang}), and more generally, for the finite quotient rings of a Dedekind domain which can seen as below.

\noindent \textbf{Theorem C.} \cite{KravitzSah}  \ {\sl Let $D$ be a Dedekind domain and $K$ a nonzero proper ideal of $D$ such that $R=D\diagup K$ is a finite ring.  Then ${\rm I}(\mathcal{S}_R)\geq {\rm D}({\rm U}(R))+\Omega(K)-\omega(K),$ where $\Omega(K)$ is the number of the prime ideals (repetitions are counted) and $\omega(K)$ the number of distinct prime ideals
in the factorization when $K$
is factored into a product of prime ideals. Moreover, equality holds for the case when
$K$ is factored into either a power of some prime ideal or a product of some pairwise distinct prime ideals.}

\noindent $\bullet$ {\small For any commutative unitary ring $R$, let ${\rm U}(R)$ be the group of units and $\mathcal{S}_R$ the multiplicative semigroup of the ring $R$.}

The author obtained a sharp lower bound for  a general finite commutative ring which generalized Theorem C.

\noindent \textbf{Theorem D.} \cite{wangAIMS}
{\sl Let $R$ be a finite commutative ring with identity. Then
$${\rm I}(\mathcal{S}_R)\geq {\rm D}({\rm U}(R))+ \sum\limits_{M} ({\rm ind}(M)-1)$$ where $M$ is taken over all distinct prime ideals of $R$ and ${\rm ind}(M)$ is the least integer $t>0$ such that $M^{t}=M^{t+1}$. Moreover, equality holds if $R$ is a local ring or all prime ideals of $R$ have the indices one.}

The above Theorem D established a connection between the Erd\H{o}s-Burgess constant of a ring $R$ and the Davenport constant of its group of units by the indices of all prime ideals of $R$. Motivated by the study of Davenport constant with weights, it would be interesting to consider the following question:

`For a finite commutative ring with identity $R$, does the above relation between the Erd\H{o}s-Burgess constant $R$ and the Davenport constant of its group of units still hold under given automorphisms of $R$?'

To consider this question, we need to generalize the Erd\H{o}s-Burgess with weights in the setting of commutative semigroups, for which the definition is as follows.

\begin{definition} Let $\mathcal{S}^*, \mathcal{S}$ be commutative semigroups, and let ${\rm hom}(\mathcal{S}^*, \mathcal{S})$ be the set of all homomorphisms from $\mathcal{S}^*$ to $\mathcal{S}$.
Let $\Psi$ be a nonempty subset of ${\rm hom}(\mathcal{S}^*, \mathcal{S})$. Define the $\Psi-$weighted-Erd\H{o}s-Burgess constant of $\mathcal{S}$, denoted ${\rm I}_{\Psi}(\mathcal{S})$, to be
the least $\ell\in \mathbb{N}\cup \{\infty\}$, such that every sequence over $\mathcal{S}^*$
of length at least $\ell$ must contain some terms $a_1,\ldots, a_{r}$ ($r\geq 1$)
such that $\prod\limits_{i=1}^r  \psi_i(a_i)$ is one idempotent of $\mathcal{S}$ where $\psi_1,\ldots,\psi_r\in \Psi$.
\end{definition}

Let $R$ be a finite commutative ring with identity. Let $\Psi$ be a subgroup of ${\rm Aut}(R)$. In this paper, we established a connection between the
 $\Psi-$weighted-Erd\H{o}s-Burgess constant of the multiplicative semigroup $\mathcal{S}_R$ and the $\Psi-$weighted Davenport constant of its group of units by a function of indices of all the prime ideals of $R$.  Our main Theorems are stated as Theorems \ref{Theorem finite commutative rings}, and \ref{Theorem Dedekind rings} which
 generalizes Theorem C.

\section{Notations}

For integers $a,b\in \mathbb{Z}$, we set $[a,b]=\{x\in \mathbb{Z}: a\leq x\leq b\}$. For a real number $x$, we denote by $\lfloor x\rfloor$ the largest integer that is less
than or equal to $x$, and by $\lceil x\rceil$ the smallest integer that is greater than or equal to $x$.

Let $R$ be a finite commutative ring with identity $1_R$. The addition and multiplication of $R$ will be denoted as $+$ and $*$ respectively.
 We use idempotent to mean an element $e\in R$ such that $e*e=e$.
Let ${\rm U}(R)$ be the group of all units of $R$.
The set ${\rm Aut}(R)$ of all ring automorphisms $R\rightarrow R$ forms a group under the operation of composition of functions.
Let ${\rm Spec} \ R$ be the spectrum of $R$, i.e., the set of all prime ideals. Let $\Psi$ be a subgroup of the group ${\rm Aut}(R)$.
Then the group $\Psi$ acts on the set ${\rm Spec} \ R$ given by:
$\psi P=\psi(P)\in {\rm Spec} \ R \  \mbox{ where } \psi\in \Psi \mbox{ and } P\in {\rm Spec} \ R.$
For any $P\in {\rm Spec} \ R$, let $$\mathscr{O}_P=\{Q\in {\rm Spec} \ R: Q=\psi P \mbox{ for some } \psi\in \Psi\}$$ be the orbit of $P$ under the action of $\Psi$ on ${\rm Spec} \ R$. Note that $$|\mathscr{O}_P|=\frac{|\Psi|}{|{\rm St}(P)|}$$ where $${\rm St}(P)=\{\psi\in \Psi: \psi(P)=P\}$$ is the stabilizer of $P$.

We also need to introduce notation and terminologies on sequences over rings and follow the notation of A. Geroldinger, D.J. Grynkiewicz and
others used for sequences over groups (cf. [\cite{Grynkiewiczmono}, Chapter 10] or [\cite{GH}, Chapter 5]).  For any nonempty subset $A$ of the ring $R$ (usually $A$ is taken to be $R$ or ${\rm U}(R)$ in this paper), let ${\cal F}(A)$
be the free commutative monoid, multiplicatively written, with basis
$A$. We denote multiplication in $\mathcal{F}(A)$ by the boldsymbol $\cdot$ and we use brackets for all exponentiation in $\mathcal{F}(A)$. By $T\in {\cal F}(A)$, we mean $T$ is a sequence of terms from $A$ which is
unordered, repetition of terms allowed. Say
$$T=a_1a_2\cdot\ldots\cdot a_{\ell}$$ where $a_i\in A$ for $i\in [1,\ell]$.
The sequence $T$ can be also denoted as $T=\mathop{\bullet}\limits_{a\in A}a^{[{\rm v}_a(T)]},$ where ${\rm v}_a(T)$ is a nonnegative integer and
means that the element $a$ occurs ${\rm v}_a(T)$ times in the sequence $T$. By $|T|$ we denote the length of the sequence, i.e., $|T|=\sum\limits_{a\in A}{\rm v}_a(T)=\ell.$ By $\varepsilon$ we denote the
{\sl empty sequence} in ${\cal F}(A)$ with $|\varepsilon|=0$. We call $T'$
a subsequence of $T$ if ${\rm v}_a(T')\leq {\rm v}_a(T)\ \ \mbox{for each element}\ \ a\in A,$ denoted by $T'\mid T,$ moreover, we write $T^{''}=T\cdot  T'^{[-1]}$ to mean the unique subsequence of $T$ with $T'\cdot T^{''}=T$.  We call $T'$ a {\sl proper} subsequence of $T$ provided that $T'\mid T$ and $T'\neq T$. In particular,  $\varepsilon$ is a proper subsequence of every nonempty sequence.
We call $T$ a {\bf $\Psi$-idempotent-product} sequence provided that there exists $\psi_1,\ldots,\psi_{\ell}\in \Psi$ (not necessarily distinct) such that $\prod\limits_{i=1}^{\ell} \psi_i(a_i)$ is an idempotent of $R$. We call $T$ a {\bf $\Psi$-idempotent-product-free} sequence provided that $T$ contains no nonempty subsequence which is $\Psi$-idempotent-product.

Note that the restriction $\psi\mid {\rm U}(R)$
is an automorphism of the group ${\rm U}(R)$. For convenience, in what follows we still use $\psi$ instead of $\psi\mid {\rm U}(R)$ to denote the automorphism of the group ${\rm U}(R)$. Then it does make a sense to say a sequence of terms from ${\rm U}(R)$ to be $\Psi-$idempotent-product or $\Psi-$idempotent-product free.

To give our main result, we still need the following notation. For any pair of positive integers $(m,h)$, let $${\rm T}(m; h)={\rm max} \sum\limits_{i=1}^h t_i, \ \ \ \ \ \ (t_i\geq 0)$$ where
\begin{equation}\label{equation iti<dh}
\sum\limits_{i=1}^d i t_i< d m \mbox{ for all } d=1,2,\ldots,h.
\end{equation}
In particular, ${\rm T}(1,h)=0$ for any positive integer $h$.

\section{Results}

We start with the following easy proposition.

\begin{prop} For any integer $h\geq 1$, there exist $t_1,\ldots,t_{h}\geq 0$
such that $\sum\limits_{i=1}^d t_i=T(m,h)$ and \eqref{equation iti<dh} holds,
where $t_1,\ldots,t_{h}$ satisfy the following recurrence relation:

(i) $t_1=m-1$;

(ii) $t_d=\left\lfloor\frac{(dm-1)-\sum\limits_{i=1}^{d-1} i t_i}{d}\right\rfloor$ for all $d=2,\ldots,h$.
\end{prop}

\begin{lemma}\label{Lemma finite commutative rings} \ Let $R$ be a finite commutative ring with identity.
 Let $\Psi$ be a subgroup of the group ${\rm Aut}(R)$. Suppose that for every prime ideal $P$ with ${\rm Ind}(P)>1$, there exists some element $c_P\in P$ such that  $(c_P)+P^{{\rm Ind}(P)}=P$.
 Then
${\rm I}_{\Psi}(R)\geq {\rm D}_{\Psi}({\rm U}(R))+\sum\limits_{P\in {\rm Spec}(R)} \frac
{T\left({\rm Ind}(P); \  \frac{|\Psi|}{|{\rm St}(P)|}\right)} {\frac{|\Psi|}{|{\rm St}(P)|}}.$
\end{lemma}

\begin{proof}   Denoted
\begin{equation}\label{equation the definition of hole}
{\rm G}_P=\{c\in P: (c)+P^{{\rm Ind}(P)}=P \},
\end{equation}
by the hypothesis of the lemma, we have
\begin{equation}\label{equation holeP is not empty}
{\rm G}_P\neq \emptyset.
\end{equation}

\noindent \textbf{Claim A.} Let $P\in {\rm Spec} \ R$ and $\psi\in \Psi$. Then,

(i) $\psi(P)^t=\psi(P^t)$ for all $t\geq 0$, and in particular, ${\rm Ind}(\psi(P))={\rm Ind}(P)$;

(ii) $\psi(G_p)=G_\psi(P)$.

\noindent {\sl Proof of Claim A.}

(i) Trivial.

(ii) Take an arbitrary element $a\in G_p$.  By \eqref{equation the definition of hole} and Conclusion (i), we have that $(\psi(a))+\psi(P)^{{\rm Ind}(\psi(P))}=(\psi(a))+\psi(P)^{{\rm Ind}(P)}=(\psi(a))+\psi(P^{{\rm Ind}(P)})=\psi((a))+\psi\left(P^{{\rm Ind}(P)}\right)=\psi\left((a)+P^{{\rm Ind}(P)}\right)=\psi(P)$. Since $\psi(a)\in \psi(G_p)\subset \psi(P)$, it follows that $\psi(a)\in G_{\psi(P)}$.  By the arbitrariness of choosing $a$, we have $\psi(G_p)\subset G_{\psi(P)}$. Since $\psi^{-1}\in \Psi$, it follows that $G_{\psi(P)}=\psi(\psi^{-1}(G_{\psi(P)}))\subset \psi(G_{\psi^{-1}(\psi(P))})= \psi(G_P)$, and thus, $\psi(G_p)=G_{\psi(P)}$. \qed



\noindent \textbf{Claim B.} Let $P\in {\rm Spec} \ R$ be such that ${\rm Ind}(P)>1$. Let $\ell\in [1,{\rm Ind}(P)-1]$ and $a_1,\ldots,a_{\ell}\in G_P$ (not necessarily distinct). Then $\prod\limits_{i=1}^{\ell} a_i \notin P^{{\rm Ind}(P)}.$

\noindent {\sl Proof of Claim B.} It suffices to consider the case that $\ell={\rm Ind}(P)-1.$
 Since $P^ {{\rm Ind}(P)}\subsetneq  P^ {{\rm Ind}(P)-1}$, we can take some element $x$ of $P^ {{\rm Ind}(P)-1}\setminus   P^ {{\rm Ind}(P)}$. Since $x$ is a finite sum of products of the form $b_{1}*b_{2}*\cdots *b_{{\rm Ind}(P)-1}$ where $b_{1}, b_{2},\ldots, b_{{\rm Ind}(P)-1}\in P$, it follows
that
\begin{equation}\label{equation sequence over Pwith}
\{T\in {\cal F}(P):  |T|={\rm Ind}(P)-1 \mbox{ and } \pi(T)\in P^ {{\rm Ind}(P)-1}\setminus   P^ {{\rm Ind}(P)}\}\neq \emptyset.
\end{equation}
Then we take a sequence $y_1\cdot \ldots\cdot y_{{\rm Ind}(P)-1} \in {\cal F}(P)$ in the set given as \eqref{equation sequence over Pwith} such that
the number of common terms of sequences $y_1\cdot \ldots\cdot y_{{\rm Ind}(P)-1}$ and $a_1\cdot \ldots\cdot a_{{\rm Ind}(P)-1}$ is {\bf maximal},
say
\begin{equation}\label{equation y1=a1,...yt=at}
y_i=a_i \ \mbox{ for each } \ i\in [1,s]
\end{equation}
(with $s \in [0,{\rm Ind}(P)-1]$ being maximal), moveover, both sequences $y_{s+1}\cdot\ldots\cdot y_{{\rm Ind}(P)-1}$ and $a_{s+1}\cdot\ldots\cdot a_{{\rm Ind}(P)-1}$ have no common terms.
To prove Claim B, we need only to show that $s={\rm Ind}(P)-1$.
Assume to the contrary that $$s<{\rm Ind}(P)-1.$$
Let $z=\prod\limits_{i\in [1,{\rm Ind}(P)-1]\setminus \{s+1\}} y_i.$
Since $y_{s+1}*z=y_{s+1}*\prod\limits_{i\in [1,{\rm Ind}(P)-1]\setminus \{s+1\}} y_i=\prod\limits_{i\in [1,{\rm Ind}(P)-1]} y_i\notin P^{{\rm Ind}(P)}$, it follows that
$y_{s+1}\in P\setminus  (P^{{\rm Ind}(P)}: z)$, equivalently,
 \begin{equation}\label{equation (PindP:b)capPinP}
 (P^{{\rm Ind}(P)}: z)\cap P\subsetneq P.
 \end{equation}
Since $a_{s+1}\in {\rm G}_P$, then $(a_{s+1})+P^{{\rm Ind}(P)}=P$. Combined with
 \eqref{equation (PindP:b)capPinP}, we conclude that $a_{s+1}\notin (P^{{\rm Ind}(P)}: z)$. It follows from \eqref{equation y1=a1,...yt=at} that
$\left(\prod\limits_{i\in [1, s+1]} a_i\right) * \left(\prod\limits_{i\in [s+2, {\rm Ind}(P)-1]}  y_i\right)=a_{s+1}* \left(\prod\limits_{i\in [1, s]} a_i\right) * \left(\prod\limits_{i\in [s+2, {\rm Ind}(P)-1]} y_i\right)=a_{s+1}* \left(\prod\limits_{i\in [1, s]} y_i\right) * \left(\prod\limits_{i\in [s+2, {\rm Ind}(P)-1]} y_i\right)=a_{s+1}* \left(\prod\limits_{i\in [1, {\rm Ind}(P)-1]\setminus \{s+1\}} y_i\right) =a_{s+1}*z\notin P^{{\rm Ind}(P)}$. Obviously, $\left(\prod\limits_{i\in [1, s+1]} a_i\right) * \left(\prod\limits_{i\in [s+2, {\rm Ind}(P)-1]}  y_i\right)\in P^{{\rm Ind}(P)-1}$, and so $\left(\prod\limits_{i\in [1, s+1]} a_i\right) * \left(\prod\limits_{i\in [s+2, {\rm Ind}(P)-1]}  y_i\right)\in P^{{\rm Ind}(P)-1}\setminus P^{{\rm Ind}(P)}$ which implies that the sequence $$a_1\cdot\ldots\cdot a_{s+1}\cdot y_{s+2}\cdot\ldots\cdot y_{{\rm Ind}(P)-1}\in \{T\in {\cal F}(P):  |T|={\rm Ind}(P)-1 \mbox{ and } \pi(T)\in P^ {{\rm Ind}(P)-1}\setminus   P^ {{\rm Ind}(P)}\}.$$ Then we derive a contradiction with the choosing of the sequence $y_1\cdot\ldots\cdot y_{{\rm Ind}(P)-1}$. This proves Claim B. \qed

Let $\mathscr{O}$ be
an arbitrary orbit of ${\rm Spec}  R$ under the action of $\Psi$.
Say,
\begin{equation}\label{equation orbit O}
\mathscr{O}=\{P_1,P_2,\ldots,P_{h}\} \ \mbox{ where } \ h=\frac{|\Psi|}{|{\rm St}(P_1)|}.
\end{equation}
For any nonempty subset $X\subset [1,h]$, we define
\begin{equation}\label{equation def H(O;X)}
H_{\mathscr{O}; X}=\left\{a\in R:  a\equiv 1_R  {\pmod {Q^{{\rm Ind}(Q)}}} \mbox{ for all }  Q\in {\rm Spec} R\setminus  \{P_i: i\in X\}\right\}\cap\left(\bigcap\limits_{i\in X}{\rm G}_{P_i}\right).
\end{equation}
For any $t\in [1,h]$, let
\begin{equation}\label{equation def H(O;t)}
\mathscr{H}_{\mathscr{O};  t}=\bigcup\limits_{\stackrel{X\subset [1,h]}{|X|=t}} H_{\mathscr{O}; X}.
\end{equation}

In the following, we shall give three claims on the properties  of $H_{\mathscr{O}; X}$, $\mathscr{H}_{\mathscr{O};  t}$ given as above.

\noindent\textbf{Claim C.}  For any element $b\in H_{\mathscr{O}; X}$, we have $\{i\in [1,h]: b\in P_i\}=X.$

\noindent {\sl Proof of Claim C.} By \eqref{equation the definition of hole} and \eqref{equation def H(O;X)}, we see that $b\in P_i$ for each $i\in X$. For any $Q\in {\rm Spec} R\setminus  \{P_i: i\in X\}$, since $b\equiv 1_R \pmod {Q^{{\rm Ind}(Q)}}$, then  $b\equiv 1_R \pmod Q$ and so $b\notin Q$, done. \qed

\noindent \textbf{Claim D.} $H_{\mathscr{O}; X}\neq \emptyset$ for any nonempty set $X\subset [1,h]$. In particular,  $\mathscr{H}_{\mathscr{O}; t}\neq \emptyset$ for each $t\in [1,h]$.

\noindent {\sl Proof of Claim D.} \ By \eqref{equation holeP is not empty},  we take an element
\begin{equation}\label{equation xi in HolePi}
c_i\in {\rm G}_{P_i} \mbox{ for each } i\in X.
\end{equation}
 Note that $\{P^{{\rm Ind}(P)}: P\in {\rm Spec} R\}$ is a family of ideals which are pairwise coprime. By the Chinese Remainder Theorem, we can find one element $a$ such that
\begin{equation}\label{equation a equiv xi mod piindPi for iin[1,k]}
a\equiv c_i \pmod {P_i^{{\rm Ind}(P_i)}} \mbox{ for each } i\in X
\end{equation}
and
\begin{equation}\label{equation a equiv 1mod Q in ClaD}
a\equiv 1_R \pmod {Q^{{\rm Ind}(Q)}} \mbox{ for each } Q\in {\rm Spec} R\setminus  \{P_i: i\in X\}.
\end{equation}
To derive $H_{\mathscr{O}; X}\neq \emptyset$, by \eqref{equation a equiv 1mod Q in ClaD}
 we need only to show that $a\in \bigcap\limits_{i\in X} G_{P_i}$.

Let $i\in X$. It follows from \eqref{equation a equiv xi mod piindPi for iin[1,k]} that
\begin{equation}\label{equation a=xi+v}
c_i=a+v_i \mbox{ for some } v_i\in P_i^{{\rm Ind}(P_i)}.
\end{equation}
Since $G_{P_i}\subset P_i$, it follows from \eqref{equation xi in HolePi} and \eqref{equation a=xi+v} that $a\in P_i$. Then it suffices to show that $(a)+P_i^{{\rm Ind}(P_i)}=P_i$.
Take an arbitrary element $b\in P_i$. By \eqref{equation the definition of hole} and \eqref{equation xi in HolePi}, we derive that $b=r_b c_i+u_b$ where $r_b \in R$ and $u_b\in P_i^{{\rm Ind}(P_i)}$. Combined with \eqref{equation a=xi+v}, we have that $b=r_b(a+v_i)+u_b=r a+(r_b v_i +u_b)\in (a)+P_i^{{\rm Ind}(P_i)}$. By the arbitrariness of choosing $b$, we proved $(a)+P_i^{{\rm Ind}(P_i)}=P_i$ and so $H_{\mathscr{O}; X}\neq \emptyset$. Then $\mathscr{H}_{\mathscr{O};t}\neq \emptyset$ follows from \eqref{equation def H(O;t)} trivially. \qed

\noindent \textbf{Claim E.}  For any $t\in [1,h]$ and any $\psi\in \Psi$, we have $\psi(\mathscr{H}_{\mathscr{O}; t})=\mathscr{H}_{\mathscr{O}; t}.$

\noindent {\sl Proof of Claim E.}  Take an arbitrary element $a\in \mathscr{H}_{\mathscr{O};t}$, equivalently, $a\in H_{\mathscr{O}; X}$ for some $X\subset [1,h]$ with $|X|=t$. Since $\psi$ acts on ${\rm Spec} R$, it follows from \eqref{equation orbit O} that there exists $X'\subset [1,h]$ of cardinality $|X'|=|X|=t$ such that
  \begin{equation}\label{equation psi(Pi)inX=PjinX'}
 \{\psi(P_i):i\in X\}= \{P_j:j\in X'\},
 \end{equation}
 and follows from Claim A (i) that \begin{equation}\label{equation Qindx=Q'index}
 \{\psi(Q)^{{\rm Ind}(\psi(Q))}:  Q\in {\rm Spec} R\setminus  \{P_i: i\in X\}\}=\{Q'^{{\rm Ind}(Q')}:  Q'\in {\rm Spec} R\setminus  \{P_j: j\in X'\}\}.
 \end{equation}
By \eqref{equation def H(O;X)} and Claim A (i),   for all $Q\in {\rm Spec} R\setminus  \{P_i: i\in X\}$, we have that $\psi(a)-1_R=\psi(a)-\psi(1_R)=\psi(a-1_R)\in \psi(Q^{{\rm Ind}(Q)})=\psi(Q)^{{\rm Ind}(Q)}=\psi(Q)^{{\rm Ind}(\psi(Q))}$, i.e.,  $\psi(a)\equiv 1_R \pmod {\psi(Q)^{{\rm Ind}(\psi(Q))}}$.
Combined with \eqref{equation Qindx=Q'index}, we conclude that $$\psi(a)\equiv 1_R \pmod {Q'^{{\rm Ind}(Q')}} \mbox{ for all } Q'\in {\rm Spec} R\setminus  \{P_j: j\in X'\}\}.$$
Since $a\in \bigcap\limits_{i\in X}{\rm G}_{P_i}$, it follows from \eqref{equation psi(Pi)inX=PjinX'} and Claim B that $\psi(a)\in  \psi(\bigcap\limits_{i\in X}{\rm G}_{P_i})=\bigcap\limits_{i\in X}\psi({\rm G}_{P_i})=\bigcap\limits_{j\in X'}{\rm G}_{P_j}.$ Then we conclude that $\psi(a)\in H_{\mathscr{O}; X'}\subset \mathscr{H}_{\mathscr{O};t}$, completing the proof of Claim D. \qed

Let $\mathscr{O}$ be the orbit given as \eqref{equation orbit O}. By $B_{\mathscr{O}}$ we denote the sequence associated with the orbit $\mathscr{O}$ which are given as below.
By Claim D, we choose an element
\begin{equation}\label{equation btin H(O,t)}
b_t\in \mathscr{H}_{\mathscr{O};t} \mbox{ for each } t\in [1,h].
\end{equation}
Let $d_1,d_2,\ldots,d_{h}$ be positive integers such that
\begin{equation}\label{equation sum di=T()}
\sum\limits_{i=1}^h d_i=T({\rm Ind}(P_1); \  h)
\end{equation}
and
\begin{equation}\label{equation tdt<n IndP1}
\sum\limits_{i=1}^n i d_i<n \ {\rm Ind}(P_1)  \mbox{ for all } n=1,2,\ldots,h.
\end{equation}
 Then we set
\begin{equation}\label{equation sequence BO}
B_{\mathscr{O}}=b_1^{[d_1]}\cdot\ldots\cdot b_h^{[d_h]}.
\end{equation}
One thing worth remarking is that all the quantities $h,d_1,d_2,\ldots,d_{h}$ in \eqref{equation sequence BO} varies with $\mathscr{O}$ taking distinct orbits, and thus,  the length  $|B_{\mathscr{O}}|$ varies accordingly. Precisely, we have the following.

\noindent \textbf{Claim F.} $|B_{\mathscr{O}}|=\sum\limits_{P\in \mathscr{O}} \frac
{T\left({\rm Ind}(P); \  \frac{|\Psi|}{|{\rm St}(P)|}\right)} {\frac{|\Psi|}{|{\rm St}(P)|}}$.

\noindent {\sl Proof of Claim F.} Note that ${\rm St}(P_1)=\cdots={\rm St}(P_h)$. By Claim A, we have ${\rm Ind}(P_1)=\cdots={\rm Ind}(P_h)$. Then it follows from \eqref{equation orbit O}, \eqref{equation sum di=T()} and \eqref{equation sequence BO} that $|B_{\mathscr{O}}|=T({\rm Ind}(P_1); \  h)=\sum\limits_{i=1}^h \frac{T({\rm Ind}(P_1); \  h)}{h}=\sum\limits_{i=1}^h \frac{T({\rm Ind}(P_1); \  \frac{|\Psi|}{|{\rm St}(P_1)|})}{\frac{|\Psi|}{|{\rm St}(P_1)|}}=\sum\limits_{i=1}^h \frac{T({\rm Ind}(P_i); \  \frac{|\Psi|}{|{\rm St}(P_i)|})}{\frac{|\Psi|}{|{\rm St}(P_i)|}}=\sum\limits_{P\in \mathscr{O}} \frac
{T\left({\rm Ind}(P); \  \frac{|\Psi|}{|{\rm St}(P)|}\right)} {\frac{|\Psi|}{|{\rm St}(P)|}}$. This proves Claim F.
\qed

\noindent \textbf{Claim G.} For any nonempty subsequence $W$ of $B_{\mathscr{O}}$, say $W=a_1\cdot\ldots\cdot a_{\ell}$, and for any $\psi_1,\ldots,\psi_{\ell}\in \Psi$ (not necessarily distinct), there exists $r\in [1,h]$ such that $\prod\limits_{i=1}^{\ell}\psi_i(a_i)\in P_r\setminus P_r^{{\rm Ind}(P_r)}$.

\noindent {\sl Proof of Claim G.} By \eqref{equation btin H(O,t)}, \eqref{equation sequence BO} and  Claim E, we have that \begin{equation}\label{equation psim(am)in i=1tohH}
\psi_m(a_m)\in \bigcup\limits_{i=1}^{h}\mathscr{H}_{\mathscr{O};i} \mbox{ for each } m\in [1,\ell].
\end{equation}
Let
$n=\sharp\{s\in [1,h]: \prod\limits_{m=1}^{\ell}\psi_m(a_m)\in P_s\}.$ Obviously, $1\leq n\leq h$.
By arranging the indices of $(P_1,\ldots,P_h)$ if necessary, we can assume without loss of generality that
\begin{equation}\label{equation productin[1ton]}
\{s\in [1,h]: \prod\limits_{m=1}^{\ell}\psi_m(a_m)\in P_s\}=[1,n].
\end{equation}
Then
\begin{equation}\label{equation each in at most n primes}
\{s\in [1,h]: \psi_m(a_m)\in P_s\}\subset [1,n] \  \mbox{ for each } m\in [1,\ell].
\end{equation}
By \eqref{equation psim(am)in i=1tohH}, \eqref{equation each in at most n primes} and Claim C, we derive that
\begin{equation}\label{equation psimamini=1tonHOi}
\psi_m(a_m)\in \bigcup\limits_{t=1}^{n}\mathscr{H}_{\mathscr{O};t} \mbox{ for each } m\in [1,\ell].
\end{equation}
By  \eqref{equation tdt<n IndP1}, \eqref{equation sequence BO}, \eqref{equation each in at most n primes}, \eqref{equation psimamini=1tonHOi}, Claim C and Claim E, we conclude that
$$\begin{array}{llll}
& &\sum\limits_{s=1}^{n} \sharp\{m\in [1,\ell]: \psi_m(a_m)\in P_s\}\\
&=&\sum\limits_{s=1}^{h} \sharp\{m\in [1,\ell]: \psi_m(a_m)\in P_s\}\\
&=&\sum\limits_{m=1}^{\ell} \sharp\{s\in [1,h]: \psi_m(a_m)\in P_s\}\\
&=& \sum\limits_{t=1}^{n} (t  \times \sharp\{m\in [1,\ell]: \psi_m(a_m)\in \mathscr{H}_{\mathscr{O};t}\})\\
&=&\sum\limits_{t=1}^n  (t  \times \sharp\{m\in [1,\ell]: a_m\in \mathscr{H}_{\mathscr{O};t}\})\\
&\leq & \sum\limits_{t=1}^n  t d_t\\
&<& n \ {\rm Ind}(P_1).\\
\end{array}$$
Combined with Claim A, we derive that there exists
\begin{equation}\label{equation rin[1ton]}
r\in [1,n]
\end{equation} such that $\sharp\{m\in [1,\ell]: \psi_m(a_m)\in P_r\}<{\rm Ind}(P_1)={\rm Ind}(P_r)$. Moreover, since $P_r$ is prime, it follows from \eqref{equation productin[1ton]} that
$\sharp\{m\in [1,\ell]: \psi_m(a_m)\in P_r\}\geq 1,$ and so, $$\sharp\{m\in [1,\ell]: \psi_m(a_m)\in P_r\}\in [1, {\rm Ind}(P_r)-1].$$
By arranging the indices of $[1,\ell]$ if necessary, we can assume without loss of generality that there exists some
\begin{equation}\label{equation uin[1,indpr-1]}
u\in [1,{\rm Ind}(P_r)-1]
\end{equation}
such that
\begin{equation}\label{equation twocasepsim(am)}
\begin{array}{llll} &\left
\{\begin{array}{llll}
               \psi_m(a_m)\in P_r, & \mbox{ if } m=1,\ldots,u; \\
               \psi_m(a_m)\notin P_r, & \mbox{ if } m=u+1,\ldots,\ell. \\
                       \end{array}
           \right. \\
\end{array}
\end{equation}
By \eqref{equation def H(O;X)}, \eqref{equation def H(O;t)}, \eqref{equation each in at most n primes},  \eqref{equation twocasepsim(am)} and Claim C, we conclude that
$$\begin{array}{llll} &\left
\{\begin{array}{llll}
               \psi_m(a_m)\in {\rm G}_{P_r}, & \mbox{ if } m=1,\ldots,u; \\
               \psi_m(a_m)\equiv 1_R \pmod{P_r^{{\rm Ind}(P_r)}}, & \mbox{ if } m=u+1,\ldots,\ell. \\
                       \end{array}
           \right. \\
\end{array}$$
It follows that $$\prod\limits_{m=1}^{\ell} \psi_m(a_m)=(\prod\limits_{m=1}^{u} \psi_m(a_m))\cdot (\prod\limits_{m=u+1}^{\ell} \psi_m(a_m))\equiv (\prod\limits_{m=1}^{u} \psi_m(a_m))\cdot 1_R=\prod\limits_{m=1}^{u} \psi_m(a_m) \pmod{P_r^{{\rm Ind}(P_r)}},$$
and
follows from \eqref{equation uin[1,indpr-1]} and Claim B that $$\prod\limits_{m=1}^{u} \psi_m(a_m) \not\equiv 0\pmod{P_r^{{\rm Ind}(P_r)}}.$$
Therefore, it follows from \eqref{equation productin[1ton]} and \eqref{equation rin[1ton]} that $\prod\limits_{m=1}^{\ell} \psi_m(a_m)\in P_r\setminus P_r^{{\rm Ind}(P_r)}$. This proves Claim G.
\qed

Take a $\Psi$-product-one free sequence $V$ of terms from the group ${\rm U}(R)$ with length
\begin{equation}\label{equation |V|_D-1}
|V|={\rm D}_{\Psi}({\rm U}(R))-1.
\end{equation}
Suppose that $\mathscr{O}_1,\ldots,\mathscr{O}_k$ are all distinct orbits of ${\rm Spec}  R$ under the action of $\Psi$. Let
\begin{equation}\label{equation T=Bo1BokV}
T= B_{\mathscr{O}_1}\cdot\ldots\cdot B_{\mathscr{O}_k}\cdot V.
\end{equation}

 \noindent \textbf{Claim H.} The sequence $T$ is $\Psi$-idempotent-product free.

\noindent {\sl Proof of Claim H.} Suppose to the contrary that there exists a nonempty subsequence $T'$ of $T$, say $T'=a_1\cdot \ldots\cdot a_{\ell}$,  such that $T'$ is a $\Psi$-idempotent-product sequence. Since $1_R$ is the unique idempotent  in ${\rm U}(R)$ and V is $\Psi$-product-one free, it follows  that at least one term of the sequence $B_{\mathscr{O}_1}\cdot\ldots\cdot B_{\mathscr{O}_k}$ appears in the sequence $T'$. Then by rearranging the indices $i\in [1,\ell]$ we may assume without loss of generality that
$$a_1\cdot\ldots\cdot a_s\mid B_{\mathscr{O}_1},$$
$$a_{s+1}\cdot\ldots\cdot a_{\ell} \mid B_{\mathscr{O}_2}\cdot\ldots\cdot B_{\mathscr{O}_k}\cdot V,$$
where $1\leq s\leq \ell$.
Since  $T'$ is a $\Psi$-idempotent-product sequence, there exists $\psi_1,\ldots,\psi_{\ell}\in \Psi$ such that $\prod\limits_{i=1}^{\ell}\psi_i(a_i)$ is an idempotent. By Claim G, there exists some prime ideal $P\in\mathscr{O}_1$ such that  $\prod\limits_{i=1}^{s}\psi_i(a_i)\in P\setminus P^{{\rm Ind}(P)}$. By \eqref{equation def H(O;X)}, we see that  $\psi_i(a_i)\equiv 1_R \pmod {P^{{\rm Ind}(P)}}$ or $\psi_i(a_i)\in {\rm U}(R)$ according to $a_i\mid B_{\mathscr{O}_2}\cdot\ldots\cdot B_{\mathscr{O}_k}$ or $a_i\mid V$ respectively, where $i\in [s+1,\ell]$.  Then $\prod\limits_{i=1}^{\ell}\psi_i(a_i)\in P\setminus P^{{\rm Ind}(P)}$ still holds. This is a contradiction with $\prod\limits_{i=1}^{\ell}\psi_i(a_i)$ being idempotent, since  $\prod\limits_{i=1}^{\ell}\psi_i(a_i)=(\prod\limits_{i=1}^{\ell}\psi_i(a_i))^{{\rm Ind}(P)}\in P^{{\rm Ind}(P)}$. This proves Claim H. \qed

By \eqref{equation |V|_D-1}, \eqref{equation T=Bo1BokV}, Claim F and Claim H,  we conclude that ${\rm I}_{\Psi}(\mathcal{S}_R)\geq 1+|T|=1+|V|+\sum\limits_{i=1}^k |B_{\mathscr{O}_i}|={\rm D}_{\Psi}({\rm U}(R))+\sum\limits_{i=1}^k\sum\limits_{P\in \mathscr{O}_i} \frac
{T\left({\rm Ind}(P); \  \frac{|\Psi|}{|{\rm St}(P)|}\right)} {\frac{|\Psi|}{|{\rm St}(P)|}}
={\rm D}_{\Psi}({\rm U}(R))+\sum\limits_{P\in {\rm Spec}(R)} \frac
{T\left({\rm Ind}(P); \  \frac{|\Psi|}{|{\rm St}(P)|}\right)} {\frac{|\Psi|}{|{\rm St}(P)|}}.$ This complete the proof of the lemma. \end{proof}

\begin{theorem}\label{Theorem finite commutative rings} \ Let $R$ be a finite commutative principal ideal ring with identity. Let $\Psi$ be a subgroup of the group ${\rm Aut}(R)$. Then
$${\rm I}_{\Psi}(\mathcal{S}_R)\geq {\rm D}_{\Psi}({\rm U}(R))+\sum\limits_{P\in spec(R)} \frac
{T\left({\rm Ind}(P); \  \frac{|\Psi|}{|{\rm St}(P)|}\right)} {\frac{|\Psi|}{|{\rm St}(P)|}}.
$$
\end{theorem}

It is not hard to check that the following proposition holds, i.e., the bound in Theorem \ref{Theorem finite commutative rings} is best possible in general.

\begin{prop} Let $L$ be a finite commutative local P.I.R, and let $R=L\times L$. Let $\varphi:R\rightarrow R$ be an automorphism given as $\varphi:(a,b)\mapsto (b,a)$ for any $(a,b)\in R$. Let $H=\langle \varphi\rangle$ be the group of automorphisms of order two generated by $\varphi$. Then
${\rm I}_{\Psi}(\mathcal{S}_R)={\rm D}_{\Psi}({\rm U}(R))+\sum\limits_{P\in spec(R)} \frac
{T\left({\rm Ind}(P); \  \frac{|\Psi|}{|{\rm St}(P)|}\right)} {\frac{|\Psi|}{|{\rm St}(P)|}}.
$
\end{prop}

It is shown \cite{Hungerford} that the following proposition holds. For the reader's convenience, we provide a short proof of this proposition.

\begin{prop} Any finite commutative principal ideal ring $R$ is a quotient ring some P.I.D.
\end{prop}

\begin{proof} By the fundamental theorem for Noetherian rings, we know that $R\cong R_1\times \cdots\times R_k$ where $R_1,\ldots, R_k$  are finite commutative local rings.  It is shown (see \cite{Hungerford}, Definition 9 and Corollary 11) that a finite commutative local P.I.R is a quotient ring of some P.I.D. Say $D_i$ is a P.I.D, $J_i\lhd D_i$ and $R_i\cong D_i\diagup J_i$ where $i\in [1,k]$. Then $R\cong (D_1\diagup J_1)\times \cdots\times (R_k\diagup J_k)\cong (D_1\times\cdots\times D_k)\diagup (J_1\times\cdots\times J_k)$, where $(J_1\times\cdots\times J_k)\lhd (D_1\times\cdots\times D_k)$. Note that $D_1\times\cdots\times D_k$ is a P.I.D. The conclusion is proved.
\end{proof}

Since a P.I.D is a Dedekind domain, therefore, any finite P.I.R is a quotient ring of some Dedekind domain.
In the following, we shall show the result of Theorem \ref{Theorem finite commutative rings} holds true for a more general setting, i.e., for finite quotient rings of any Dedekind domain.

\begin{theorem}\label{Theorem Dedekind rings} \ Let $R$ be a finite quotient ring of some Dedekind domain.
Let $\Psi$ be a subgroup of the group ${\rm Aut}(R)$. Then
${\rm I}_{\Psi}(\mathcal{S}_R)\geq {\rm D}_{\Psi}({\rm U}(R))+\sum\limits_{P\in spec(R)} \frac
{T\left({\rm Ind}(P); \  \frac{|\Psi|}{|{\rm St}(P)|}\right)} {\frac{|\Psi|}{|{\rm St}(P)|}}.
$
\end{theorem}

\begin{proof} Let $Q\in {\rm Spec} R$ be such that ${\rm Ind}(Q)>1$. Then there exists no ideal $A$ of $R$ such that $Q^2\subsetneq A \subsetneq Q$. Let $R=D\diagup J$ where $D$ is a Dedekind domain and $J\lhd D$.
Let $\varphi: D\rightarrow R$ be the canonical epimorphism. Then
\begin{equation}\label{equation kernal J=}
J=\prod\limits_{i=1}^t P_i^{\alpha_i}
\end{equation}
 where $t\geq 1, \alpha_1,\ldots, \alpha_t\geq 1$,  and  $P_1,\ldots,P_t$ are distinct prime ideals of $D$. It follows from \eqref{equation kernal J=} that
 \begin{equation}\label{equation Q=varphi(P)}
 Q=\varphi(P)
 \end{equation}
  for some $P\in {\rm Spec} D$ with
  \begin{equation}\label{equation J subset P}
  J\subset P.
  \end{equation}
   Then $P=P_i$ for some $i\in [1,t]$, say $$P=P_1.$$

 \noindent \textbf{Claim I.} There exists no ideal $N\lhd R$ such that $Q^2\subsetneq N\subsetneq Q$.

 \noindent {\sl Proof of Claim I.} Assume to the contrary that there exists some ideal $N\lhd R$ such that $Q^2\subsetneq N\subsetneq Q$.
By \eqref{equation Q=varphi(P)} and \eqref{equation J subset P}, we have that $\varphi(P^2+J)=\varphi(P^2)\subset \varphi(P)*\varphi(P)=Q^2$, and thus,   $P^2\subset P^2+J=\varphi^{-1}(\varphi(P^2+J))\subset \varphi^{-1}(Q^2)\subsetneq\varphi^{-1}(N) \subsetneq \varphi^{-1}(Q)=P$. Then we derive a contradiction, since $D$ is a Dedekind domain implying that there exists no ideal $M\lhd D$ such that $P^2\subsetneq M\subsetneq P$. This proves Claim I.  \qed

Take an arbitrary $Q\in {\rm Spec R}$ such that ${\rm Ind}(Q)>1$. Take an element $x\in Q\setminus Q^2$.
Since $Q^2\subsetneq (x)+Q^2\subset Q$, it follows from Claim I that $Q=(x)+Q^2$.
Then we have that
\begin{equation}\label{equation interationforx+q2}
(x)+Q^{k}=(x)+((x)+Q^2)^k=(x)+\left(\sum\limits_{i=0}^{k-1}(x)^{k-i}*Q^{2i}\right)+Q^{2k}=(x)+Q^{2k} \ \mbox{ for any } \ k\geq 1.
\end{equation}
Fix an integer $m>\ln {\rm Ind}(Q)$. It follows from \eqref{equation interationforx+q2} that $Q=(x)+Q^2=(x)+Q^4=\cdots=(x)+Q^{2^m}=(x)+Q^{{\rm Ind}(Q)}$. Then the conclusion follows from Lemma \ref{Lemma finite commutative rings} readily.
\end{proof}

Then we close the paper with the following two conjectures.

\begin{conj}\label{Conjecture} Let $R$ be a finite ring with identity.
Let $\Psi$ be a subgroup of the group ${\rm Aut}(R)$. Then
${\rm I}_{\Psi}(\mathcal{S}_R)\geq {\rm D}_{\Psi}({\rm U}(R))+\sum\limits_{P\in spec(R)} \frac
{T\left({\rm Ind}(P); \  \frac{|\Psi|}{|{\rm St}(P)|}\right)} {\frac{|\Psi|}{|{\rm St}(P)|}}.
$
\end{conj}

\begin{conj}\label{Conjecture} Let $R$ be a finite P.I.R with identity.
Let $\Psi$ be a subgroup of the group ${\rm Aut}(R)$. Then
${\rm I}_{\Psi}(\mathcal{S}_R)={\rm D}_{\Psi}({\rm U}(R))+\sum\limits_{P\in spec(R)} \frac
{T\left({\rm Ind}(P); \  \frac{|\Psi|}{|{\rm St}(P)|}\right)} {\frac{|\Psi|}{|{\rm St}(P)|}}.
$
\end{conj}

\noindent {\bf Acknowledgements}

This work is supported by NSFC (grant no. 11971347).

\end{document}